\documentclass{amsart}
\usepackage{amsmath, amsthm, amssymb, latexsym}
\usepackage{graphicx}
\usepackage{amsfonts}
\usepackage{eucal}
\usepackage{amscd}
\usepackage{amssymb}
\usepackage{xypic}
\usepackage[all]{xy}

\author{J. Vandehey}
\title{On multiplicative functions with bounded partial sums}
\date{\today}

\newtheorem{thm}{Theorem}[section]

\newtheorem{lem}[thm]{Lemma}

\begin{document}

\maketitle

\begin{abstract}

Consider a multiplicative function $f(n)$ taking values on the unit circle.  Is it possible that the partial sums of this function are bounded?  We show that if we weaken the notion of multiplicativity so that $f(pn)=f(p)f(n)$ for all primes $p$ in some finite set $P$, then the answer is yes.  We also discuss a result of Bronstein that shows that functions modified from characters at a finite number of places.

\end{abstract}

%%%%%%%%%%%%%%%%%%%%%%%%%%
\section{Introduction}
%%%%%%%%%%%%%%%%%%%%%%%%%%

In number theory, quite commonly, we wish to understand the rate of growth of the partial sums of a multiplicative function $f(n)$: $$S(x)=S_f(x)=\sum_{n\le x} f(n).$$
In the case that $|f(n)| \le 1$, many theorems, especially those of Delange \cite{D1} \cite{D2}, Hall \cite{Hall}, and Halasz \cite{Halasz}, give upper bounds on the rate of growth.

The question of lower bounds on the rate of growth of $S(x)$ are less well-known.  It may be that $|S(x)|$ oscillates greatly, and so one is interested in finding a function $g(x)$ such that $|S(x)| \neq o(g(x))$ -- that is, $|S(x)| = \Omega(g(x))$.

Erdos and Chudakoff \cite{E1} \cite{E2} \cite{E3} independently asked the following, more general, question: given $f(n)$ with values in $\{-1,1\}$ (not necessarily multiplicative), is it true that for every $c$, there exist $d$ and $m$ so that $$\left| \sum_{k=1}^m f(kd) \right| >c?$$  This has since become known as the Erdos discrepancy problem (EDP), and Erdos offered \$500 for a proof or counter-example.

Recently, the EDP was the subject of a Polymath project, lead by Timothy Gowers.  Though the problem itself remained elusive, many interesting ideas were put forth, and we suggest the interested reader visit the Polymath website\footnote{http://michaelnielsen.org/polymath1/} for more information.

There are many variants of the EDP, but in this paper we are interested in the variant where $f(n)$ is multiplicative and takes values in $\mathbb{T}:=\{ z\in\mathbb{C}: |z|=1 \}$.  This is informally known as the Multiplicative Erdos discrepancy Problem (MEDP), as in \cite{C}.  

In an attempt to solve the MEDP, Kevin Ford suggested the following problem in private correspondence: Let $P$ be a subset of the prime numbers, and say that a function $f(n)$ is $P$-multiplicative if $f(pn)=f(p)f(n)$ for all $p\in P$.  Do there exist $P$-multiplicative functions $f:\mathbb{N}\to \mathbb{T}$ such that $|S_f(x)|=O(1)$? If not, then clearly no completely multiplicative function could have bounded partial sums, and this would solve the MEDP.  However, in this paper we prove the following theorem:

\begin{thm}
For every finite set of primes $P$, there exist $P$-multiplicative functions $f:\mathbb{N}\to \{-1,1\}$ with bounded partial sums.
\end{thm}

Paradoxically, the proof of this theorem gives more credence to the idea that completely multiplicative functions should not have bounded partial sums.  The particular $P$-multiplicative functions constructed in the proof have bounded partial sums, but as we add more elements to $P$, the bound on the sum appears to tend to infinity.

If we suppose that the partial sums of multiplicative functions are unbounded, then another interesting question would be to understand how slowly they could tend towards infinity.  The ideas of pretentiousness of Granville and Soundararajan \cite{GS1} \cite{GS2} say that if a multipliactive function $f(n)$ ``pretends'' to be another multiplicative function $g(n)$, in the sense that their values on the primes are close to each other, then one expects their partial sums to share similar characteristics.  This suggests that a multiplicative function with slowly growing partial sums might ``pretend'' to be a character, which truly does have bounded partial sums.  So, we say a function in $\mathbb{F}(\mathbb{T} \cup \{0\})$ is character-like if there exists a Dirichlet character $\chi(n)$ with conductor $P$ such that $f(n)=\chi(n)$ when $(n,P)=1$, and $f$ is itself not a Dirichlet character.

The following result was proved first by Bronstein \cite{B}; however, it appears to have been forgotten in the modern literature.  A weaker result appeared in a recent paper of Borwein, Choi, and Coons \cite{BCC}, and the current author, initially unaware of Bronstein's original paper, proved the same result, albeit by a less efficient method.

\begin{thm}
If $f$ is character-like then $$S_f(x)=\Omega(\log(x)).$$
\end{thm}

It is curious that Bronstein's proof, that of Borwein, Choi, and Coons, and that of the author's followed such similar methods, despite being done independantly.  The proof can be quickly sketched as follows:
\begin{itemize}
\item Pick a prime $p|P$ and let $P'$ be the largest divisor of $P$ relatively prime to $p$.  Consider $$S'(x):=\sum_{d|P'} \mu(d) S_f(x/d) = \sum_{n\le x} f'(n)$$ where $f'$ is the multiplicative function defined by $f'(p)=f(p)$ on all primes not dividing $P'$ and $0$ elsewhere.  Assume $p$ was chosen so $f'$ is also not a Dirichlet character.  It is clear that if $S'(x)=\Omega(\log(x))$, then $S_f(x)=\Omega(\log(x))$.
\item Find an integer $k$ such that $S'(k P')\neq 0$.\footnote{It is in this step that Bronstein's proof was more efficient than the author's: Bronstein's $k$ was smaller.}
\item Construct an integer value $x = \sum_{i=1}^n k P' p^{m_i}$ for some appropriate choice of $m_i$'s and show that $S'(x) \gg n \gg \log(x)$.
\end{itemize} 

In particular, the first step reduces the problem to a function which differs from a character at only one prime, and the final step relies on this simplification, as $S'(x)$ is dependent solely on the base $P$ expansion of $x$.  In general, for any character-like $f(n)$, $S_f(x)$ can be written as a sum of at most $(\log x)^{\omega(P)}$ character sums, so has size at most $O((\log x)^{\omega(P)})$.  How close to this the lower bound can be is still an open question.

\section{Proof of the theorem}

Consider a set of primes $p_1< p_2< p_3<\ldots< p_k$, $k > 2$, and let $P= \prod_{1 \le i \le k} p_i$.  We wish to have an arithmetic function $f(n)$ such that $f(n) =\pm 1 $, $f(p_i n) = f(p_i) f(n)$ for $1 \le i \le k$, and $S_N := \sum_{n \le N} f(n)$ is bounded.

We will begin by demonstrating two specific cases.

\subsection{The case $k=2$, $p_1=2$, $p_2=3$}

Let $f(n)$ be a function such that $f(n) =\pm 1 $, $f(2 n) = f(2) f(n)$ and $f(3n)=f(3)f(n)$.  Such a function is completely determined by its values at $f(2)$, $f(3)$, and $f(6m+1),f(6m+5)$, $m \ge 0$. We will show that we can assign the values of $f$ at these points to create a function that is bounded.

First, let $f(1):=1$, $f(2):=1$, $f(3):=-1$, $f(5)=-1$.  Note that $$f(6n+1)+f(6n+2)+f(6n+3)=-(f(6n+4)+f(6n+5)+f(6n+6))=\pm 1\qquad (*)$$ for $n=0$.  I claim that if $f$ satisfies $(*)$ for all $n < N$ then we can specify values for $f(6N+1)$, $f(6N+5)$ so that the relation is true for $N$ as well.  Note that the values of $f(6N+2)$, $f(6N+3)$, $f(6N+4)$, $f(6N+6)$ are all predetermined by earlier values.

In particular, if $f(6N+2)=f(6N+3)$ then let $f(6N+1):=-f(6N+2)$; and if $f(6N+4)=f(6N+6)$ then let $f(6N+5):=-f(6N+4)$.  (These conditions guarantee $f(6N+1)+f(6N+2)+f(6N+3)=\pm 1$ and $f(6N+4)+f(6N+5)+f(6N+6)=\pm 1$.)  Now we have several cases to consider:
\begin{enumerate}
\item Case 1: if $f(6N+2)=f(6N+3)=-f(6N+4)=-f(6N+6)$, then the assigned values to $f(6N+1),f(6N+5)$ guarantee that $(*)$ holds.
\item Case 2: if $f(6N+2)=f(6N+3)$ but $f(6N+4)\neq f(6N+6)$, then let $f(6N+5):=-f(6N+2)$, which will guarantee that $(*)$ holds.
\item Case 3: if $f(6N+2)\neq f(6N+3)$ but $f(6N+4)= f(6N+6)$, then let $f(6N+1):= -f(6N+4)$, which will guarantee that $(*)$ holds.
\item Case 4: if $f(6N+2)\neq f(6N+3)$ and $f(6N+4)\neq f(6N+6)$, then let $f(6N+1):=1$, $f(6N+5)=-1$, which will guarantee that $(*)$ holds.
\item Case 5: the relation $f(6N+2)=f(6N+3)=f(6N+4)=f(6N+6)$ is impossible since $f(6N+2)+f(6N+4)+f(6N+6)=f(2)(f(3N+1)+f(3N+2)+f(3N+3))$, which, by assumption, is $\pm 1$.
\end{enumerate}

With all the needed values for $f(n)$ defined inductively in this way, we see that $$\sum_{n \le 6m} f(n) = 0$$ and hence $$\sum_{n\le x} f(n) = O(1).$$

\subsection{The case $k=3$, $p_1=2$, $p_2=3$, $p_3=5$}

This case is considerably more complicated.  Let $f(n)$ be a function such that $f(n) =\pm 1$ and $f(p_i n) = f(p_i) f(n)$ for $i=1,2,3$.  Such a function is again completely determined by its values at $f(2)$, $f(3)$, $f(5)$, and $f(30m+k)$ with $m \ge 0$ and $k=1,7,11,13,17,19,23,29$.

Suppose that $f(2):=-1$, $f(3):=-1$, $f(5):=1$, and for all $m\ge 0$ let the following relations hold: \begin{align*}f(30m+1):= -f(30m+3)&\qquad f(30m+7):=-f(30m+5) \\ f(30m+11):=-f(30m+9)&\qquad  f(30m+13):=-f(30m+15)\\ f(30m+19):=-f(30m+21)&\qquad  f(30m+23):=-f(30m+25)\\ f(30m+29)&:=-f(30m+27).\end{align*}

From this we can see that $\sum_{n=1}^{15} f(30m+(2n-1)) = f(30m+17)$ for $m\ge 0$.  $f(30m+17)$ is the only term of the form $f(30m+k)$ with $k=1,7,11,13,17,19,23,29$ whose value we have not yet specified.

It would be nice if most of the terms in $\sum_{n=1}^{30} f(30m+n)$ cancel with one another, as indeed, most of the odd terms do; while in this particular set-up, we have a lot of cancellation among these terms, in the more general case below, this will not be guaranteed, so we will assume there is little we can say about them.  Instead, we look at a larger interval.

Take $\sum_{n=1}^{60} f(30(2m)+n)$.  In this, we have the sum over all odd terms, except the terms $f(30(2m)+17)$ and $f(30(2m+1)+17)$, is zero, and all the terms congruent to $2$ modulo $4$ also sum to zero if we again exclude the term $f(30(2m)+2\cdot 17)$.  Again, we will assume there is little we can say about the terms divisible by $4$.  To get some additional cancellation, let us suppose that $f(30(2m)+17):=-f(30(2m)+2\cdot 17)$ for all $m\ge 0$.  Unless $m=0$ this defines $f(30(2m)+17)$ in terms of $f(30m+17)$, a much earlier term in the sequence, and when $m=0$, this merely says that $f(17):= - f(2) f(17)$ which holds since $f(2)=-1$.

Looking at the even longer sum $\sum_{n=1}^{120} f(30(4m)+n)$, all the odd terms except $f(30(4m+j)+17)$, $1\le j \le 4$, sum to zero.  All the terms divisible by 2 except for $f(30(4m+2j)+2\cdot 17)$, $1 \le j \le 2$ sum to zero.  All the terms divisible by $4$ except for $f(30(4m)+4\cdot 17)$ sum to zero as well.  However, our relation $f(30(2m)+17):=-f(30(2m)+2\cdot 17)$ implies automatically that $f(30(4m)+4\cdot 17)=-f(30(2m)+2\cdot 17)$.  We can visually represent the cancellations occurring with the following picture:
$$\xymatrix@!0@C=45pt{
30(4m)+17\ar@{-}[rd] & & 30(4m+1)+17\ar@{--}[ld] & & 30(4m+2)+17\ar@{-}[rd] & & 30(4m+3)+17\ar@{--}[ld]\\
 & 30(4m)+2\cdot17\ar@{-}[rrd] & & & & 30(4m+1)+2\cdot17\ar@{--}[lld] & \\
 & & & 30(4m)+4\cdot17
}$$
Bold lines represent that the value of the two connected terms have values that cancel with one another, dashed lines just help to show relative placement.

If we were to double the length of the interval again, we would obtain a diagram that looks like
$$
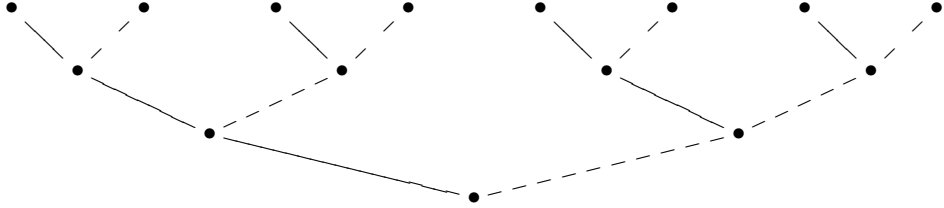
\xymatrix@!0@C=25pt
{
{\bullet}\ar@{-}[rd] & & {\bullet}\ar@{--}[ld] & & {\bullet}\ar@{-}[rd] & & {\bullet}\ar@{--}[ld]& & {\bullet}\ar@{-}[rd] & & {\bullet}\ar@{--}[ld] & & {\bullet}\ar@{-}[rd] & & {\bullet}\ar@{--}[ld]\\
 & {\bullet}\ar@{-}[rrd] & & & & {\bullet}\ar@{--}[lld] & & & & {\bullet}\ar@{-}[rrd] & & & & {\bullet}\ar@{--}[lld] & \\
 & & & {\bullet}\ar@{-}[rrrrd] & & & & & & & & {\bullet}\ar@{--}[lllld]  \\
 & & & & & & & {\bullet}
}$$
Consider the sum of each set of numbers linked by bold lines.  If there are an even number of terms linked together, then they sum to zero.  If there are an odd number of terms, then the value of the sum is equal to the value of the term in the linked set that is on the top level of the diagram.  So let us rewrite our diagram again, writing down only the top level of the previous diagram with numbers to indicate how many other numbers it is linked to:
$$\xymatrix{
*+[F]{4} & *+[o][F]{1} & *+[F]{2} & *+[o][F]{1} & 3 & 1\ar[l] & *+[F]{2} & *+[o][F]{1}
}$$
Here the boxed terms represent sums that cancel out completely, due to an even number of terms.  The arrow between the 1 and 3 represents that we want the relation $f(30(8m+5)+17):=-f(30(8m+4)+17)$ in order to cause the set of 3 terms in the middle to sum to zero.  The circled terms represent terms of the form $f(30(8m+j)+17)$, $j=1,3,7$ that we have not yet specified any values for.

However, this is still not enough and we must double the length of our interval 3 more times, now looking at an interval of length $64\cdot 30=1920$ to obtain the following diagram:
$$\xymatrix@C=10pt{
{7} & {1}\ar[l] & *+[F]{2} & *+[o][F]{1} & 3 & 1\ar[l] & *+[F]{2} & *+[o][F]{1} & *+[F]{4} & *+[o][F]{1} & *+[F]{2} & *+[o][F]{1} & 3 & 1\ar[l] & *+[F]{2} & *+[o][F]{1} \\
{5} & {1}\ar[l] & *+[F]{2} & *+[o][F]{1} & 3 & 1\ar[l] & *+[F]{2} & *+[o][F]{1} & *+[F]{4} & *+[o][F]{1} & *+[F]{2} & *+[o][F]{1} & 3 & 1\ar[l] & *+[F]{2} & *+[o][F]{1} \\
*+[F]{6} & *+[o][F]{1} & *+[F]{2} & *+[o][F]{1} & 3 & 1\ar[l] & *+[F]{2} & *+[o][F]{1} & *+[F]{4} & *+[o][F]{1} & *+[F]{2} & *+[o][F]{1} & 3 & 1\ar[l] & *+[F]{2} & *+[o][F]{1} \\
{5} & {1}\ar[l] & *+[F]{2} & *+[o][F]{1} & 3 & 1\ar[l] & *+[F]{2} & *+[o][F]{1} & *+[F]{4} & *+[o][F]{1} & *+[F]{2} & *+[o][F]{1} & 3 & 1\ar[l] & *+[F]{2} & *+[o][F]{1} \\
}$$

The arrows between the 1 and 5 corresponds to a relation $f(30(32m+17)+17):=-f(30(32m+16)+17)$, and the arrow between the 1 and 7 corresponds to a relation $f(30(64m+1)+17):=-f(30(64m)+17)$.

The circled $1$'s correspond to values $f(30(64m+j)+17)$, $j=3$, $7$, $9$, $11$, $15$, $19$, $23$, $25$, $27$, $31$, $33$, $35$, $39$, $41$, $45$, $47$, $51$, $55$, $57$, $59$, $63$, all of which we have not specified any values for.  Let us index these values of $j$, calling them $j_i$ with $j_1=3$, $j_2=7$ and so on.

Now if we consider the sum $\sum_{n=1}^{64\cdot 30} f(30(64m)+n)$  We see that all terms where the power of $2$ that divides the argument is at most $64$ sum to zero, except for the terms $f(30(64+j_i)+17)$. The terms divisible by $128$ have arguments of the form $l_i:=30(64m)+128i$ for $1 \le i \le 16$.  To make the sum in question equal to zero, we set $$f(30(64m+j_i)+17):=-f(l_i) \qquad \text{for }1 \le i \le 16$$ and then let $$f(30(64m+j_i)+17):=(-1)^i\qquad \text{for }16 \le i \le 21$$

If all the values of $f(30(64m+j)+k)$ with $0 \le m < N$, $0 \le j < 64 m$, and $k=1,7,11,13,17,19,23,29$ satisfy all the relations given above, then we can assign the values of $f(30(64N+j)+k)$ so that they also satisfy all the relations given above.  As noted above, with our particular value of $f(2)$, the values of $f(30(64\cdot 0 +j)+k)$ can be chosen to satisfy all the relations, and hence, we can define all the values of $f(30m+k)$ inductively.

With these values, we see that $$\sum_{n \le 30\cdot 64 m} f(n) = 0$$ and hence $$\sum_{n\le x} f(n) = O(1).$$

\subsection{The general case}
To generate a function with these properties, we require the use of several intermediate steps.  Let $F_{k-1},F_{k-2},\ldots,F_1 \subset \mathbb{N}$ be non-empty sets whose elements are relatively prime to $P$, $R_{k-1},R_{k-2},\ldots,R_1 \subset \mathbb{N}^2$, and $b_{k-1},b_{k-2},\ldots,b_1\in \mathbb{N}$ that all satisfy the following conditions:
\begin{enumerate}\renewcommand{\theenumi}{\Roman{enumi}}

\item $F_{j-1} \subset F_j$, and if $n \in F_j$, then $n+b_j \in F_j$ and (provided $n>2 b_j$) $n-b_j \in F_j$;
\item If $(x_1,y),(x_2,y) \in R_j$, then $x_1 = x_2$;
\item $R_{j-1}\supset R_j$, and if $(x,y)\in R_{j-1}\setminus R_j$, then $y \in F_j \setminus F_{j-1}$ and $\lfloor \frac{x}{b_{j-1}} \rfloor = \lfloor \frac{y}{b_{j-1}} \rfloor$;
\item For $1 \le j \le k-2$, $b_j = p_j^{a_j} b_{j+1}$ for some $a_j \in \mathbb{N}$, and $b_{k-1}=P$;
\item Suppose $f(n)$ is any arithmetic function such that $f(p_in)=f(p_i)f(n)$ for $j \le i \le k$ and $f(x)=-f(y)$ for $(x,y)\in R_j$, and let $$S_{N,j} : = \sum_{\substack{n < N \\ p_i \nmid n \text{ for }1 \le i < j \\ n \not\in F_j}} f(n).$$  Then $S_{b_j (M+1),j}-S_{b_jM+1,j} = 0$ for all $M \in \mathbb{N}$.
\end{enumerate}

$R_j$ represents a set of $R$elations that we expect $f$ to satisfy.  $F_j$ is a set of $F$ree elements that we can use to build new relations from.  $b_j$ is a useful modulus that we work with.

We start with the construction of $R_{k-1},F_{k-1}$ by letting $P'= \prod_{1 \le i \le k-2} p_i$.  Then over any interval $I_M:=[b_{k-1}M+1,b_{k-1}(M+1)]$, $\frac{\phi(P')}{P'}P = \phi(P')p_{k-1}p_k$ of the numbers in $I_M$ are relatively prime to $P'$. $\phi(P)$ of the numbers in $I_M$ are relatively prime to all the $p_k$.  Since $\left(1- \frac{1}{p_{k-1}} \right)\left(1- \frac{1}{p_{k}} \right) \ge\left(1- \frac{1}{3} \right)\left(1- \frac{1}{5} \right) >\frac{1}{2}$, there are more numbers relatively prime to $P$ in $I_M$ than there are relatively prime to $P'$ but divisible by $p_{k-1}$ or $p_k$ in $I_M$.

Suppose $f(n)$ is an arithmetic function such that $f(p_i n) = f(p_i) f_k(n)$ for $i = k-1$ or $k$ and consider $$\sum_{\substack{n \in I_M \\ p_i \nmid n \text{ for }1 \le i < j}} f(n).$$  Looking at the terms, we see that some $f(n)$ are ``fixed'' because $n$ is divisible by some $p_i$ with $k-2 < i \le k$, and their value - and hence the value of $f(n)$ - is predetermined by earlier terms.  The rest are ``free'' and we can specify relations $(x,y)\in I_M^2$ to be placed in $R_{k-1}$ with the $x$ from the fixed values and the $y$ from the free values.  By the previous paragraph, we have more free values in such an interval than fixed so all the fixed values can be canceled in this way.  This gives $S_{P (M+1),k-1}-S_{PM+1,k-1} = 0$, if $M\ge1$ for any function that satisfies these new relations.

To be more specific, let $$U_M:=\{n \in I_M|(n,P)=1   \}$$ and $$V_M:=\{n\in I_M | (n,P')=1, \ p_k|n \text{ or } p_{k-1}|n  \},$$ with elements $u_i$ and $v_i$, respectively, arranged in ascending order.  Let $F_M^{(0)}=\{u_i|i\le |V_M|\}$ and $R_M^{(0)}=\{(v_i,u_i)|i\le |V_M|\}.$  Finally, let $F_{k-1}=\bigcup_{M=1}^\infty U_M\setminus F_M^{(0)}$ and $R_{k-1}=\bigcup_{M=1}^\infty R_M^{(0)}$.  These sets satisfy conditions $(I)-(V)$ with $b_{k-1}=P$.

This shows that there exists at least one choice of $F_{k-1}$, and $R_{k-1}$, with $b_{k-1}=P$.  The following lemma will imply that we have at least one choice of $F_1$, $R_1$, and $b_1$ as well.

\begin{lem}
If $F_j$, $R_j$, $b_j$, $j>1$ are as above, then there exists $F_{j-1}$, $R_{j-1}$, and $b_{j-1}=p_{j-1}^{a_{j-1}} b_j$, that satisfy the above requirements as well.
\end{lem}

\begin{proof}

First, define $B:=|[b_j+1,2b_j]\cap F_j|$.

Then, let $a_{j-1}$ be large enough so that $$B \sum_{r=0}^{a_{j-1}-1} (-1)^r p_{j-1}^{a_{j-1}-1-r} > \frac{b_{j}}{p_{j-1}}.$$

Again, consider the intervals $I_M:=[b_{j-1}M+1,b_{j-1}(M+1)]$ starting with $I_1$ and then, in turn, $I_2$, $I_3$, and so on.  For each $I_M$, we work in three steps, creating new sets $F_M^{(1)}, F_M^{(2)}, F_M^{(3)}$ and $R_M^{(1)}, R_M^{(2)}, R_M^{(3)}$.

\begin{enumerate}

\item Consider all terms of the form $p_{j-1}b_jm+t \in I_M \cap F_j$ with $0 < t < b_j$.  Let $F_M^{(1)}$ be the set that consists of all such $p_{j-1}b_jm+t$.  Let $R_M^{(1)}$ be the set that consists of all $(p_{j-1}(b_jm+t),p_{j-1}b_jm+t)$.

 \item Consider all $$n' = p_{j-1}^sb_jm+t \in I_M \cap F_j$$ with $0 < t < b_j$, such that either $1\le s < a_{j-1}$ and $(m,p_{j-1})=1$ or else $s=a_{j-1}$.  Let $F_M^{(2)}$ be the set that consists of all such $(p_{j-1}^s m+1) b_j +t$ with $s$ even.  And let $R_M^{(2)}$ be the set that consists of all $(p_{j-1}^sb_jm+t,(p_{j-1}^s m+1) b_j +t)$ again with $s$ even.

\item Let $$V=\left\{v_i:=b_{j-1}M+ip_{j-1}^{a_{j-1}+1}\quad | \quad 1< i\le b_j/p_{j-1}  \right\}$$ denote the set of $n \in I_M$ for which $p_{j-1}^{a_{j-1}+1} | n$.
    The number of $n' \in I_M \cap F_j$ in the previous step with $s=a_{j-1}-i$ is, for each $t$, equal to $(p_{j-1}-1)p_{j-1}^{i-1}$ if $0 < i < a_{j-1}$ and at least $1$ if $i=0$.  There are exactly $B$ distinct $t$. So, the number of terms in $I_M\cap F_j \setminus F_M^{(1)}\setminus F_M^{(2)}$ is at least the number of $n'$ with $s$ odd, and this is at least $$B\sum_{r=0}^{a_{j-1}-1} (-1)^r p_{j-1}^{a_{j-1}-1-r}.$$

Hence, we have, for $a_{j-1}$ as defined at the start of the proof, that $$|I_M \cap F_j \setminus F_M^{(1)} \setminus F_M^{(2)}| > |V|.$$  Now, let $u_1 < u_2 < \ldots < u_{|V|}$ be the $|V|$ smallest terms in $I_M\cap F_{j-1}\setminus F_M^{(1)} \setminus F_M^{(2)}$.  Let $F_M^{(3)}$ be the set consisting of all the $u_i$'s.  And let $R_M^{(3)}$ be the set consisting of the elements $(v_i,u_i)$ for $1 \le i \le |V|$.

\end{enumerate}

Now let $F_{j-1}=F_j \setminus \bigcup_{M=1}^\infty (F_M^{(1)} \cup F_M^{(2)} \cup F_M^{(3)})$ and $R_{j-1}=R_j\cup \bigcup_{M=1}^\infty (R_M^{(1)} \cup R_M^{(2)} \cup R_M^{(3)})$.

It is clear, by construction, that the $F_{j-1}$, $R_{j-1}$ produced in this way will satisfy conditions (I) through (IV).  It remains to check condition (V).

Now let $f$ be any arithmetic function that satisfies $f(p_in)=f(p_i)f(n)$ if $j \le i \le k$, and $f(x)+f(y)=0$ if $(x,y)\in R_{j-1}$.

Consider the sum $$T=\sum_{\substack{n \in I_M \\ p_i \nmid n \text{ for }1 \le i < j-1 }} f(n),$$ which, by (V) for $F_j,$ $R_j$, equals $$\sum_{\substack{n \in I_M \\n \in F_{j}}} f(n)+\sum_{\substack{n \in I_M \\ p_{j-1}|n}} f(n).$$

Fix an $s$ such that $1 \le s\le a_{j-1}$.  Consider the set $J_M^s$ of terms $n\in I_M$ such that $p_{j-1}^s||n$, $n/p_{j-1}^s=b_j m +t \not\in F_j$ and are coprime to $p_1 p_2 \cdots p_{j-1}$.  Then $J_M^s/p_{j-1}^s$ consists of the set of numbers coprime to $p_1 p_2 \cdots p_{j-1}$ that are contained in $$[(b_{j-1}M+1)/p_{j-1}^s,b_{j-1}(M+1)/p_{j-1}^s] \setminus F_j,$$ which equals $$[p_{j-1}^{a_{j-1}-s} b_jM+1,p_{j-1}^{a_{j-1}-s} b_j(M+1)]\setminus F_j$$ since we only consider the integer points within intervals.

Thus $$\sum_{n\in J_M^s} f(n) = f(p_{j-1})^s \sum_{n\in J_M^s/p_{j-1}^s} f(n).$$ By condition (V), however, the latter sum must be 0.  So, $$T=\sum_{\substack{n \in I_M \\ n\in F_j }} f(n)+\sum_{\substack{n \in I_M \\ p_{j-1}^{a_{j-1}+1}|n}} f(n)+ \sum_{\substack{ n\in I_M \\ p_{j-1}^s||n \qquad 1\le s \le a_{j-1}\\ n/p_{j-1}^s \in F_j}} f(n).$$

Now consider an $n=p_{j-1}^{\tilde{s}}(b_j\tilde{m}+t)\in I_M$ with $1\le \tilde{s}\le a_{j-1}$ and $b_j\tilde{m}+t \in F_j$, a term from the third sum above.  Let $n'=p_{j-1}^{\tilde{s}}b_j\tilde{m}+t=p_{j-1}^{s}b_jm+t$ where either $s \le a_{j-1}$ and $p_{j-1}\nmid m$ or else $s=a_{j-1}$; that is, $n'$ has the same form as the elements in step (2) of the algorithm above.

Then since $(p_{j-1}(p_{j-1}^{s-1}b_jm+t),p_{j-1}^sb_jm+t) \in R_{j-1}$, we have $f(p_{j-1}^s b_j m +t) + f(p_{j-1}(p_{j-1}^{s-1} b_j m+t))=0$.

Note that $p_{j-1}^{s-1} b_j m+t \equiv n \pmod{b_j}$ and $p_{j-1}^{s-1} b_j m+t>b_j$ so $p_{j-1}^{s-1} b_j m+t \in F_j$ as well by condition (I).  Thus $(p_{j-1}(p_{j-1}^{s-2}b_jm+t),p_{j-1}^{s-1}b_jm+t) \in R_{j-1}$ and $f(p_{j-1}^{s-1} b_j m +t) + f(p_{j-1}(p_{j-1}^{s-2} b_j m+t))=0$.

Continuing in this way, we see that the sum $$f(p_{j-1}^s b_j m +t) + f(p_{j-1}(p_{j-1}^{s-1} b_j m+t)) +\cdots + f(p_{j-1}^s(b_jm+t))\qquad (*)$$ equals $0$ if $s$ is odd and $f(n')$ if $s$ is even.  The term $n=p_{j-1}^{\tilde{s}}(b_j\tilde{m}+t)$ must appear as one of the arguments a sum of the form $(*)$.

The first multiple of $p_{j-1}^{a_{j-1}}b_j$ after $p_{j-1}^s b_j m +t$ must occur at least $p_{j-1}^sb_j -t$ integers later.  However, $$p_{j-1}^s(b_jm+t)-p_{j-1}^s b_j m +t=(p_{j-1}^s-1)t<(p_{j-1}^s-1)b_j<p_{j-1}^sb_j-t.$$  Since the right endpoints of the intervals $I_M$ are all multiples of $p_{j-1}^{a_{j-1}}b_j$, this implies that every argument in a sum of the form $(*)$ is in an interval $I_M$ (for some $M$).

Thus, for every $n=p_{j-1}^{\tilde{s}}(b_j\tilde{m}+t)\in I_M$ with $1\le \tilde{s}\le a_{j-1}$ and $b_jm+t \in F_j$ we can find the corresponding $n'$ and find a sum of the form $(*)$ that contains it as an argument.  Each $p_{j-1}b_jm+t\in F_M^{(1)}$ appears as the argument of the leading term in such a sum because $p_{j-1}(b_jm+t)$ takes the form $n=p_{j-1}^{\tilde{s}}(b_j\tilde{m}+t)\in I_M$.

At this point, by using the evaluation of the sum $(*)$, we have \begin{align*}T=&\sum_{\substack{n \in I_M \\ n\in F_{j-1}\cup F_M^{(2)}\cup F_M^{(3)} }} f(n) +\sum_{\substack{n \in I_M \\ p_{j-1}^{a_{j-1}+1}|n}} f(n) \\ &+  \sum_{\substack{ n'\in I_M \\ n'=p_{j-1}^sb_jm+t\text{ as in step (2)}}} f(n').\end{align*}

But by the construction of $R_M^{(2)}$, all terms in the first sum where $n \in F_M^{(2)}$ cancel with all the terms in the third sum, we therefore have $$T=\sum_{\substack{n \in I_M \cap F_{j-1} }} f(n)+\sum_{\substack{n \in I_M \\ p_{j-1}^{a_{j-1}+1}|n }} f(n)+\sum_{\substack{n \in F_M^{(3)}}} f(n).$$  And by the construction of $R_M^{(3)}$, all terms in the second and third sum cancel, so we have that $$T=\sum_{n \in I_M \cap F_{j-1}} f(n)$$ and hence $S_{P (M+1),k-1}-S_{PM+1,k-1} = 0$.

\end{proof}

By iterating this lemma, we eventually arrive at $F_1$, $R_1$, and $b_1$.  Now suppose $F_1={a_1,a_2,a_3,\ldots}$ with $a_i < a_{i+1}$ and let $f(n)$ be any function such that $f(p_in)=f(p_i)f(n)$ if $1 \le i \le k$, and that $f(x)=-f(y)$ if $(x,y)\in R_1$ or $(x,y)=(a_i,a_{i+1})$, $i \in \mathbb{N}$.  Then

\begin{align*}S(N)&=\sum_{n \le N} f(n)\\ &=O(1)+\sum_{\substack{n \le N \\ n \in F_1 \text{ or } n \le b_1}} f(n) \\&= O(1)\end{align*}

Such a function exists for all possible choices of $f(n)$, $1 \le n \le P$.

This completes the proof of the theorem.

\end{document}